\documentclass[10pt,reqno]{amsart}
\usepackage{bbm}
\usepackage{a4wide}
\usepackage{dsfont}            
\usepackage{amssymb,amsmath,amsfonts}
\usepackage{mathrsfs}
\usepackage{graphicx} 
\usepackage{xcolor}          
\usepackage[utf8]{inputenc}
\usepackage[all]{xy}

\usepackage[colorlinks,
             linkcolor=blue,
             citecolor=green,
             pdfproducer={pdfLaTeX},
             pdfpagemode=None,
             bookmarksopen=true
             bookmarksnumbered=true]{hyperref}
             
\usepackage[capitalise]{cleveref}

\crefname{section}{Section}{Sections}
\crefformat{section}{#2Section~#1#3} 
\Crefformat{section}{#2Section~#1#3} 

\crefname{subsection}{\S}{\S\S}
\crefformat{subsection}{#2\S#1#3} 
\Crefformat{subsection}{#2\S#1#3} 

\crefname{definition}{definition}{definitions}
\crefformat{definition}{#2definition~#1#3} 
\Crefformat{definition}{#2Definition~#1#3} 

\crefname{ex}{example}{examples}
\crefformat{example}{#2example~#1#3} 
\Crefformat{example}{#2Example~#1#3} 

\crefname{remark}{remark}{remarks}
\crefformat{remark}{#2remark~#1#3} 
\Crefformat{remark}{#2Remark~#1#3} 

\crefname{convention}{convention}{conventions}
\crefformat{convention}{#2convention~#1#3} 
\Crefformat{convention}{#2Convention~#1#3}

\crefname{lemma}{lemma}{lemmas}
\crefformat{lemma}{#2lemma~#1#3} 
\Crefformat{lemma}{#2Lemma~#1#3} 

\crefformat{proposition}{#2proposition~#1#3} 
\Crefformat{proposition}{#2Proposition~#1#3} 

\Crefformat{corollary}{#2Corollary~#1#3}

\crefname{theorem}{theorem}{theorems}
\crefformat{theorem}{#2theorem~#1#3} 
\Crefformat{theorem}{#2Theorem~#1#3} 

\crefname{assumption}{assumption}{Assumptions}
\crefformat{assumption}{#2assumption~#1#3} 
\Crefformat{assumption}{#2Assumption~#1#3} 

\crefname{equation}{}{}
\crefformat{equation}{(#2#1#3)} 
\Crefformat{equation}{(#2#1#3)}

\newtheorem{proposition}{Proposition}[section]
  \newtheorem{theorem}[proposition]{Theorem}
  \newtheorem{corollary}[proposition]{Corollary}
  \newtheorem{lemma}[proposition]{Lemma}
    
\theoremstyle{definition}
  \newtheorem{definition}[proposition]{Definition}
  \newtheorem{remark}[proposition]{Remark}
   
  \newtheorem{example}[proposition]{Example}
   
\allowdisplaybreaks

\newcommand{\cst}{\ifmmode\mathrm{C}^*\else{$\mathrm{C}^*$}\fi}

\newcommand\bZ{\mathbb Z}

\newcommand\cA{\mathcal A}
\newcommand\cB{\mathcal B}
\newcommand\cC{\mathcal C}

\newcommand\cN{\mathcal N}

\newcommand\cV{\mathcal V}

\newcommand{\kk}{\mathbbm{k}}

\newcommand{\CC}{\mathbb{C}}
\newcommand{\RR}{\mathbb{R}}

\newcommand{\id}{\mathrm{id}}

\newcommand{\I}{\mathds{1}}

\newcommand{\ww}{\mathrm{W}}

\DeclareMathOperator{\C}{C}

\numberwithin{equation}{section}

\def\labelitemi{$\blacktriangleright$}
\author{Alexandru Chirvasitu}
\address{Department of Mathematics, University at Buffalo, Buffalo, NY 14260-2900, USA}
\email{achirvas@buffalo.edu}
\author{Pawe{\l} Kasprzak}
\address{Department of Mathematical Methods in Physics, Faculty of Physics, University of Warsaw, Poland}
\email{pawel.kasprzak@fuw.edu.pl}
\author{Piotr Szulim}
\address{Faculty of Physics, University of Warsaw, Poland}
  \email{pw.szulim@student.uw.edu.pl}
\title{Integrals in left coideal subalgebras and group-like projections}



\begin{document}

\begin{abstract}We develop a theory of right group-like projections in Hopf algebras linking them with the theory of left coideal subalgebras with two sided counital integrals. Every right group-like projection is associated with a left coideal subalgebra, maximal among the ones containing the given group-like projection as an integral, and we show that that subalgebra is finite dimensional.  We observe that in a semisimple Hopf algebra $H$ every left coideal subalgebra has an integral and we prove a 1-1 correspondence between right group-like projections and left coideal subalgebras of $H$. We provide a number of equivalent conditions for a right group-like projections to be left group-like projection and prove a 1-1 correspondence between semisimple left coideal subalgebras preserved by the squared antipode and two sided group-like projections.

  We also classify left coideal subalgebras in Taft Hopf algebras $H_{n^2}$ over a field $\kk$, showing that the automorphism group splits them into \begin{itemize}
  \item a class of cardinality $|\kk|-1$ of semisimple ones which correspond to right group-like projections which are not two sided;
  \item finitely many semisimple singletons, each corresponding to two sided group-like projection; the number of those singletons for $H_{n^2}$ is equal to the number of divisors of $n$;
  \item finitely many singletons, each non-semisimple and admitting no right group-like projection; the number of those singletons for $H_{n^2}$ is equal to the number of divisors of $n$;
  \end{itemize} In particular we answer the question of Landstad and Van Daele showing that there do exist right group-like projections which are not left group-like projections.
\end{abstract}

\maketitle

\setcounter{tocdepth}{1}
\tableofcontents

\newlength{\sw}
\settowidth{\sw}{$\scriptstyle\sigma-\text{\rm{weak closure}}$}
\newlength{\nc}
\settowidth{\nc}{$\scriptstyle\text{\rm{norm closure}}$}
\newlength{\ssw}
\settowidth{\ssw}{$\scriptscriptstyle\sigma-\text{\rm{weak closure}}$}
\newlength{\snc}
\settowidth{\snc}{$\scriptscriptstyle\text{\rm{norm closure}}$}
\renewcommand{\labelitemi}{$\bullet$}
\section{Introduction}

In order to explain the classical motivation standing behind the concept of group like projection let us fix a finite group $G$ and consider  Hopf algebra of complex functions  $\C(G)$  on $G$. It is easy to check that given a subgroup $H\subset G$, its indicator function $\I_H\in \C(G)$ is a projection satisfying \begin{equation}\label{def_gl}\Delta(\I_H)(\I_G\otimes\I_H) = \I_H\otimes \I_H.\end{equation} Conversely,  a non-zero  projection $P\in\C(G)$ satisfying $\Delta(P)(\I\otimes P) = P\otimes P$ is of the form  $P=\I_H$. In particular    $\Delta(P)(\I\otimes P) = P\otimes P$ if  and only if $P = \I_H$ if and only if $\Delta(P)(P\otimes\I) = P\otimes P$. It turns out that  there are Hopf algebras admitting an element $P$ satisfying \[\Delta(P)(\I\otimes P) = P\otimes P = (\I\otimes P)\Delta(P)\] but not satisfying \[\Delta(P)(P\otimes \I) = P\otimes P= (P\otimes \I)\Delta(P),\]  thus we use the terminology of right and left group-like projections for the elements satisfying the former or the latter equation respectively.

A right  group-like projection $P$  in a Hopf algebra $\mathcal{A}$ cannot, in general, be matched with a quotient Hopf subalgebra, thus we do not have a full analogy with the classical case. Remarkably however, $P$ can be matched with a left coideal subalgebra denoted by $\cN_P\subset \cA$, which  is the analog of the algebra of functions on the quotient space $G/H$ from the classical context.

$\cN_P$ is characterized as being a maximal among all left coideal subalgebras $\cN\subset\cA$ containing $P$ and such that $Px = xP = \varepsilon(x)P$ for all $x\in\cN$. Denoting   $\cV_P = \{(\mu\otimes \id)(\Delta(P)):\mu\in\cA\}$  we clearly have $\cV_P\subset \cN_P$.  Having proved    that $\cN_P$ is finite dimensional (c.f. \Cref{cor.fd}) allowed us  to use \cite{Kopp} to conclude that  $\cN_P = \cV_P$ if and only if $\cN_P$ is quasi-Frobenius, \Cref{qf_fun}. In particular this holds   in the following three cases 
\begin{itemize}
    \item $P$ is two sided group-like projection, \Cref{coipr};
    \item $\cN_P$ is semisimple, \Cref{le.n=np};
    \item $\cA$ is weakly finite, \Cref{th.wf-hopf}. 
\end{itemize}
 We also prove that that  $\cN_P = \cV_P$ if  and only if $\I_{\cA}\in\cV_P$. 

The semisimplicity is the property which turns out to link nicely with the theory of group-like projections in the following two fundamental ways:
\begin{itemize}
    \item if $\cA$ is a semisimple Hopf algebra then we have   a 1-1 correspondence between left coideal subalgebras of $\cA$ and right group-like projections in $\cA$, see \Cref{th.proj-coid};
    \item $P\mapsto \cN_P$ restricts to a one two one correspondence between two sided group-like projections in $\cA$ and  semisimple left coideals subalgebras $\cN\subset \cA$ preserved by the square of the antipode, see \Cref{th.2sd-s2}. 
\end{itemize}
Using \cite{Kopp} we also provide a necessary and sufficient condition for $\cN_P$ to be semisimple, see \Cref{th.np-ss}. 
The results described above are contained in \Cref{Sec_gen}.

In \Cref{se.tft}   we classify left  coideals subalgebras in Taft Hopf   algebras $H_{n^2}$ over a field $\kk$, where   $n\in\mathbb{N}$. It turns our that the automorphism group of $H_{n^2}$  splits them into   
\begin{itemize}
    \item a class of cardinality $|\kk|-1$ of semisimple ones which  correspond to  right group-like  projections which are not two sided; 
    \item  a  finite number of semisimple  singletons  each  corresponding to two sided group-like projection; the number of those singletons  for $H_{n^2}$ is equal to the number of divisors of $n$;  
   \item  finite number of singletons each being  non-semisimple and   not admitting a right group-like projection; the number of those singletons  for $H_{n^2}$ is equal to the number of divisors of $n$.
\end{itemize}  
Every coideal and every group-like projection has an explicit description.
Our classification  enabled us to answer   the  question  of Landstad and Van Daele showing that there do exist   right group-like projections which are not left group-like projections (already in $4$-dimensional Taft algebra $H_4$).

In \Cref{se.pre}   we define left, right and two sided group-like projections together with their shifts and we prove a few preliminary results. 

For background on Hopf algebras we refer to \cite{mont,Radford_book}. We shall freely use the concept of semisimple and cosemisimple Hopf algebra, coradical filtration, primitive elements etc. It is  assumed that the reader is  familiar with the basics of the theory of modules over rings, see e.g. \cite{lam-lec}. We shall make occasional use of such concepts as injective modules, faithfully flat algebras over subalgebras, quasi-Frobenius algebras and weakly finite algebras.

The study initiated in this paper is continued  in \cite{Kasp_coint} where the theory of generalized integrals and cointegrals on a left coideal subalgebra is developed. In particular the finite dimensionality result of  \Cref{cor.fd}   holds also in the case when a given left coideal subalgebra $\cN\subset \cA$  equipped with a non-zero multiplicative functional $\mu\in\cN$ admits a non-zero $\Lambda\in\cN$ satisfying $a\Lambda = \mu(a)\Lambda$ for all $a\in\cN$, \cite[Theorem 3.7.]{Kasp_coint}. The issue of existence of such $\Lambda\in\cN$ turns out to be a non-trivial matter, linked e.g. with the   existence of generalized  cointegrals on $\cN$.  
\section{Preliminaries}\label{se.pre}

We shall assume that $(\mathcal{A},\Delta,S,\varepsilon)$ is a Hopf algebra (over an arbitrary field $\kk$) with an invertible antipode. We shall often write that $\mathcal{A}$ is a Hopf algebra, having in mind that comultiplication $\Delta$, coinverse $S$ and counit $\varepsilon$ are fixed. We shall use the Sweedler's notation: $\Delta(x) = x_{(1)}\otimes x_{(2)}$.  A vector subspace $\cV\subset \cA$ is said to be a left coideal if $\Delta(\cV)\subset \cA\otimes \cV$.  A unital subalgebra $\cN\subset \mathcal{A}$ which is a left coideal is said to be a left coideal subalgebra.  A Hopf surjection $\pi:\cA\to \cB$ can be assigned with a left coideal subalgebra $\cN_\pi =\{x\in\cA:(\id\otimes \pi)(\Delta(x)) = x\otimes\I_\cB\}$. It is easily seen that $\cN_\pi$ is preserved by the adjoint action $\textrm{Ad}:\cA\otimes \cA\to \cA$ of $\cA$ on itself, where $\textrm{Ad}(a\otimes x) = a_{(1)}xS(a_{(2)})$.  A left coideal subalgebra $\cN$ is said to be normal if the adjoint action restricts to $\mathcal{A}\otimes \cN\ni a\otimes x\mapsto a_{(1)}xS(a_{(2)})\in\cN$. It is well known that when $\cA$ is left faithfully flat over a normal left coideal subalgebra then $\cN=\cN_\pi$ for a faithfully coflat Hopf projection $\pi$, unique in the sense that its kernel is uniquely determined by $\cN$.

Let $(\cC,\Delta_{\cC},\varepsilon)$ be a coalgebra and let $\pi:\cA\to\cC$ be surjective map of coalgebras. If in addition $\cC$ is equipped with a left $\cA$-module structure $\cA\otimes\cC\to \cC$ which is a coalgebra map and $\pi:\cA\to\cC$ is $\cA$ linear then we say that $\cC$ is a left module coalgebra quotient. Given a left module coalgebra quotient $\pi:\cA\to\cC$ we assigned to it a left coideal subalgebra of $\cA$ \[\cN_\pi=\{x\in\cA:(\id\otimes\pi)(\Delta(x)) = x\otimes\pi(\I)\}.\] Conversely given a left coideal subalgebra $\cN$ we assign to it a left module coalgebra quotient $\pi:\cA\to\cC_{\cN}$ where $\cC_{\cN} = \cA/(\cA\cN^-)$ and $\cN^- = \cN\cap\ker\varepsilon$. If $\cA $ is left $\cN$-faithfully flat then $\cN_{\cC_{\cN}} = \cN$. Let us move on to the main subject of this paper.
\begin{definition}
Let $\mathcal{A}$ be a Hopf algebra and $P\in \mathcal{A}$ a non-zero projection. We say that 
\begin{enumerate}
    \item $P$ is a right group-like projection if 
    \[\Delta(P)(\I\otimes P) = P\otimes P = (\I\otimes P)\Delta(P)\]
    \item $P$ is a left group-like projection if 
    \[\Delta(P)(P\otimes\I) = P\otimes P = (P\otimes \I)\Delta(P)\]
    \item $P$ is a two sided group-like projection (or simply a group-like projection) if it is a left and right group-like projection. 
\end{enumerate}
\end{definition}
In the course of the paper we shall use the following operators on $\cA\otimes\cA$  (the formulas below are extended  to an arbitrary element of $\cA\otimes \cA$ be linearity):
\begin{equation}\label{wdef}
\begin{aligned}
\ww_{lr}(x\otimes y)&= x_{(1)}\otimes x_{(2)}y,\\
\ww_{ll}(x\otimes y)&= x_{(1)}\otimes yx_{(2)},\\
\ww_{rl}(x\otimes y)&= xy_{(1)}\otimes y_{(2)},\\
\ww_{rr}(x\otimes y)&= y_{(1)}x\otimes y_{(2)}.
\end{aligned}
\end{equation}
In our notation $\ww_{-,-}$ the first index establishes the position of the tensorand to which the comultiplication is applied whereas the second index corresponds to  the position of the remaining tensorand   in the multiplication. The operators $\ww_{-,-}$ are invertible and we have 
\begin{equation}\label{wdefinv}
\begin{aligned}
\ww^{-1}_{lr}(x\otimes y)&= x_{(1)}\otimes S(x_{(2)})y\\
\ww^{-1}_{ll}(x\otimes y)&= x_{(1)}\otimes yS^{-1}(x_{(2)})\\
\ww^{-1}_{rl}(x\otimes y)&= xS(y_{(1)})\otimes y_{(2)}\\
\ww^{-1}_{rr}(x\otimes y)&= S^{-1}(y_{(1)})x\otimes y_{(2)}.
\end{aligned}
\end{equation}
Note that $P\in\cA$ is a right group-like projection if and only if \[\ww_{lr}(P\otimes P) = P\otimes P = \ww_{ll}(P\otimes P).\] Similarly $P$ is a left group-like projection if and only if 
\[\ww_{rr}(P\otimes P) = P\otimes P = \ww_{rl}(P\otimes P).\]

As mentioned in the introduction, if $G$ is a finite group and $P\in\C(G)$ is a group like projection then $P=\I_H$ where $H\subset G$ is a subgroup of $G$. The projection $Q$ entering the next lemma is the analog of $\I_{Hg}$ for some $g\in G$. 
\begin{lemma}\label{basic_lem}
Let $y\in\mathcal{A}$ and $Q\in\mathcal{A}$ be a non-zero projection such that 
\begin{equation}\label{gleq}(\I\otimes Q)\Delta(Q) = \Delta(Q)(\I\otimes Q)= y\otimes Q.\end{equation} Then $y$  is a right group-like projection. Moreover if $S^2(Q) = Q$ then $S^2(y) = y$. 
\end{lemma}

\begin{proof} 
The identity  \begin{equation}\label{left2}(\I\otimes y) \Delta(y) = y\otimes y \end{equation}
is the consequence of the following computation
\begin{align*}
  ((\I\otimes y)\Delta(y))\otimes Q & =(\I\otimes y\otimes\I)(\Delta\otimes \id)(y\otimes Q)\\ & =(\I\otimes y\otimes Q)(\Delta\otimes\id)(\Delta(Q)) \\&= (\I\otimes\I\otimes Q)(\I\otimes\Delta(Q))(\id\otimes\Delta)(\Delta(Q))\\&=  (\I\otimes\I\otimes Q)(\id\otimes\Delta)( (\I\otimes Q)\Delta(Q))\\&=  (\I\otimes\I\otimes Q)(\id\otimes\Delta)( y\otimes Q)\\&=(y\otimes y)\otimes Q.
\end{align*}
Similarly we check that $\Delta(y)(\I\otimes y) = y\otimes y$. 
 Since  $\left((\I\otimes Q)\Delta(Q)\right)^2 = (\I\otimes Q)\Delta(Q)$ we see that  $(Q\otimes y)$ and thus $y$ are projections.
\end{proof}
\begin{remark} \label{remls}
There is a left version of \Cref{basic_lem} with $Q$ and $y\in\mathcal{A}$  satisfying $\Delta(Q)(Q\otimes\I) = (Q\otimes\I) \Delta(Q) = Q\otimes y$ and $y$ being a left group-like projection. 
In particular if $P\in\mathcal{A}$ is a right  group-like projection admitting a non-zero  $Q\in\mathcal{A}$ satisfying $\Delta(Q)(Q\otimes\I) = Q\otimes P = (Q\otimes\I) \Delta(Q)$ then $P$ is a two sided group-like projection.  
\end{remark}
\begin{definition}\label{shifts_def}
Let $P\in\mathcal{A}$ be a right group-like projection. We say that a projection $Q\in\mathcal{A}$ is a right  shift  of $P$ if 
\begin{equation}\label{QPcond1}\Delta(Q)(\I\otimes Q) = P\otimes Q = (\I\otimes Q) \Delta(Q).\end{equation} 

Let $P$ be a left group-like projection. We say that a projection $Q\in\mathcal{A}$ is a left   shift  of $P$ if 
\[\Delta(Q)(Q\otimes \I) = Q\otimes P = (Q\otimes\I) \Delta(Q).\]
\end{definition}
\begin{lemma}\label{lemma_shifts}
Let $P,Q\in\mathcal{A}$ be non-zero projections. The following conditions are equivalent 
\begin{itemize}
    \item $Q$ is a right shift of a right group-like projection $P$;
    \item $P$ and $Q$ satisfy $(\I\otimes S^{-1}(Q))\Delta(P) = Q\otimes S^{-1}(Q)$ and $\Delta(P) (\I\otimes S(Q)) = Q\otimes S(Q) $.
\end{itemize}
In particular if $Q$ is a right shift of $P$ then there exists a functional $\mu\in\mathcal{A}^*$ such that $Q = (\id\otimes\mu)(\Delta(P))$. 
Moreover $S^2(Q) = Q$ if and only if $S^2(P) = P$.
\end{lemma}
The condition $S^2(P)=P$ for a right  group-like projection $P$ turns out to be equivalent with $P$ being a two sided group-like projection, c.f. \Cref{th.summary}. 
\begin{proof}[Proof of \Cref{lemma_shifts}]
In order to get the equivalence of the first and the second bullet point it is enough to  write  \Cref{QPcond1} in the form 
\[\ww_{lr}(Q\otimes Q) = P\otimes Q = \ww_{ll}(Q\otimes Q)\] and use the inverse formulas \Cref{wdef}. 
The existence of $\mu\in\mathcal{A}^*$ satisfying $(\id\otimes\mu)(\Delta(P)) =Q$ is  obvious. 

If $S^2(Q) = Q$ then $(S^2\otimes S^2)(\Delta(Q)(\I\otimes Q)) = \Delta(Q)(\I\otimes Q)$ which implies that  $S^2(P)\otimes Q = P\otimes Q$ and thus $S^2(P) = P$. For the converse implication suppose that $S^2(P) = P$.  Applying  $S^2\otimes S^2$ to the equality  \[Q\otimes S^{-1}(Q) = (\I\otimes S^{-1}(Q))\Delta(P)(\I\otimes S^{-1}(Q))\] we get $S^2(Q)\otimes S(Q) = (\I\otimes S(Q))\Delta(P)(\I\otimes S(Q)) = Q\otimes S(Q)$ and we conclude that $S^2(Q) = Q$. 
\end{proof}

\section{General results}\label{Sec_gen}
In the first part of this section we shall derive a number of equivalent conditions for a right group-like projection to be two sided, c.f. \Cref{Sec_two_sided}. In the second part \Cref{Sec_left_coideal}  we shall relate the theory of (right) group-like projections with the theory of left  coideal subalgebras  with integrals developed in \cite{Kopp}. The third subsection \Cref{subsec_misc} is concerned with the normality, faithful flatness and other properties related with left coideals subalgebras assigned to right group-like projections. Finally in the fourth subsection \Cref{se.ss-coss} we discuss right group-like projections in a semisimple Hopf algebra $\cA$. 
\subsection{Two sided group-like projection}\label{Sec_two_sided}
 Let us begin with the following auxiliary lemma.
\begin{lemma}\label{abprop} Let $\cA$ be a Hopf algebra 
and $a,b\in\mathcal{A}$. The following two conditions are equivalent
\begin{itemize}
\item $\Delta(a)(\I\otimes b) = \Delta(b)(a\otimes \I)$
\item $(b\otimes\I)(\id\otimes S)(\Delta(a)) = a\otimes b$
\end{itemize}
Moreover if one of these equivalent conditions holds then we have $ab = \varepsilon(a)b=S(a)\varepsilon(b)$. 
\end{lemma}
\begin{proof}Suppose that $\Delta(a)(\I\otimes b) = \Delta(b)(a\otimes \I)$ holds. 
Applying $(\varepsilon\otimes \id)$ to the identity \[a_{(1)}\otimes a_{(2)}b   = b_{(1)}a\otimes b_{(2)}\] we get $ab = \varepsilon(a)b$. Moreover we have \begin{align*}\varepsilon(a)b &= S(a_{(1)})a_{(2)}b \\&= S(a)S(b_{(1)})b_{(2)} = S(a)\varepsilon(b)\end{align*}
where the second equality was obtained by applying $m\circ (S\otimes\id)$ to the identity 
$a_{(1)}\otimes a_{(2)}b  =b_{(1)}a\otimes b_{(2)}$. 
 
In order to see the equivalence of the first and the second bullet point we can write the first as $\ww_{lr}^{-1}\ww_{rr}(a\otimes b)=a\otimes b$ and note that $\ww_{lr}^{-1}\ww_{rr}(x\otimes y) = yx_{(1)}\otimes S(x_{(2)})$. 
\end{proof}

\begin{corollary}\label{Corr1}
Let $P\in\mathcal{A}$ be a non-zero projection such that 
\begin{equation}\label{H1}\Delta(P)(\I\otimes P) = \Delta(P)(P\otimes \I)\end{equation}Then $\varepsilon(P) = \I$,  $S(P) = P$ and $(P\otimes\I)\Delta(P) =  \Delta(P)(P\otimes \I) = P\otimes P$. In particular \begin{itemize}
\item every projection $P$ satisfying \Cref{H1} is a two sided  group-like projection;
\item every two sided group-like projection $P$ satisfies $S(P)=P$.
\end{itemize}
\end{corollary}

\begin{proof}
 Putting $a=b=P$ in \Cref{abprop} we get $P^2 = \varepsilon(P)P$ thus $\varepsilon(P) =1 $ and   $S(P) = \varepsilon(P)S(P) = P$. This together with  $(P\otimes\I)(\id\otimes S)(\Delta(P)) = P\otimes P$ (c.f. the second bullet point of \Cref{abprop}) yields $(P\otimes\I)\Delta(P) = P\otimes P$ . Applying $S\otimes S$ to the last identity we get $P\otimes P = \Delta(P)(\I\otimes P) = \Delta(P)(P\otimes\I)$. 
\end{proof}
 It turns out the condition $S^2(P) = P$ implies $S(P) = P$ if $P$ is a right group-like projection. More generally  the following  holds. 
\begin{lemma}\label{lems2}
Let $P\in\mathcal{A}$ be a non-zero projection such that \begin{equation}\label{H2}\Delta(P)(\I\otimes P) =P\otimes P\end{equation}
If $S^2(P) = P$ then $S(P) = P$.  
\end{lemma}
\begin{proof}
Applying $(\varepsilon\otimes\id)$ to \eqref{H2} we get $\varepsilon(P) = 1$. Next, applying $m\circ(S\otimes\id)$ to \eqref{H2} we get $P = S(P)P$. Thus $S(P) = S(P)S^2(P) = S(P)P = P$. 
\end{proof}
Let us give one more condition guaranteeing a right group-like projection to be preserved by $S$.
\begin{lemma}
Let $P\in\mathcal{A}$ be a right group-like projection.
Then the following equivalence holds:
\[\Delta(P)(\I\otimes P) = \Delta(P)(P\otimes \I) \Leftrightarrow S(P) = P\] 
\end{lemma}
\begin{proof}
Clearly if $\Delta(P)(\I\otimes P) = \Delta(P)(P\otimes \I)$ then $S(P) = P$ (see \Cref{Corr1}). Conversely, if $P$  is a right group-like projection such that $S(P) = P$ then \[P\otimes P =(S\otimes S)((\I\otimes P)\Delta(P)) = \Delta^{\textrm{op}}(P)(\I\otimes P) \] and we see that $\Delta^{\textrm{op}}(P)(\I\otimes P) = \Delta(P)(\I\otimes P)$. Applying the flip $\sigma$, where $\sigma(a\otimes b) = b\otimes a$, to both sides and using $\sigma(\Delta(P)(\I\otimes P)) = \Delta(P)(\I\otimes P)$ we get $\Delta(P)(\I\otimes P) = \Delta(P)(P\otimes \I)$.
\end{proof}

Note that if $P$ is a right group-like projection then  $P = S(P)P$ and $S(P) = PS(P)$ (see the proof of \Cref{lems2}). This observation together with  other results of this subsection yields:
\begin{theorem}\label{th.summary}
Let $P\in\mathcal{A}$  be a non-zero projection. The following conditions are equivalent:
\begin{itemize}
    \item $P$ is a right group-like projection such that $S(P) = P$;
    \item $P$ is a right group-like projection such that $S^2(P) = P$;
    \item $P$ is a right group-like projection such that $PS(P) = S(P)P$;
    \item $P$ is a left group-like projection such that $S(P) = P$;
    \item $P$ is a left group-like projection such that $S^2(P) = P$;
    \item $P$ is a left  group-like projection such that $PS(P) = S(P)P$;
    \item $P$ is a two sided group-like projection;
    \item $P$ satisfies $\Delta(P)(\I\otimes P) = \Delta(P)(P\otimes\I)$.
\end{itemize}
\end{theorem}
Using \Cref{th.summary} we can characterize two sided group-like projections it terms of their shifts, c.f. the first part  of \Cref{lrshift}.   The second part of this proposition corresponds to the classically trivial observation, that a right shift  $\I_{Hg}$  of of (two sided) group-like projection $\I_{H}$  is a left shift of (two sided)  group-like projection $\I_{g^{-1}Hg}$. Except generalizing this   fact,   \Cref{lrshift} will also be essential in the proof of \Cref{square}. 
\begin{proposition}\label{lrshift}
Let $P$ be a right group like projection and let $Q$ be a right shift of $P$ in the sense of \Cref{shifts_def}. Then $P$ is a two sided group-like projection if and only if $S^2(Q) = Q$. In the latter  case there exists a two sided group-like projection $\tilde{P}$ such that $Q$ is a left shift of $\tilde{P}$. \end{proposition}

\begin{proof}[Proof of \Cref{lrshift}]
Note that the first part of the theorem follows immediately from \Cref{th.summary} and \Cref{lemma_shifts} so let us move on to the proof of the second part. 
 Using the identities entering the second bullet point of \Cref{lemma_shifts} we can check that \begin{equation}\label{con1}QS(P_{(2)}),S(P_{(2)})Q\in\kk Q.\end{equation} 
Applying $S\otimes S$ to these identities  and using $S(P) = P$ and $S^2(Q) = Q$  we get 
\begin{equation}\label{idenlocal}Q\otimes Q = QP_{(1)}\otimes S(P_{(2)})= P_{(1)}Q\otimes S(P_{(2)}) \end{equation} which implies that  \begin{equation}\label{con2}QP_{(1)}\in\kk Q,P_{(1)}Q\in\kk Q.\end{equation} Applying $\ww_{rr}$ to the first equality in \Cref{idenlocal} we get   $Q_{(1)}Q\otimes Q_{(2)} = S(P_{(3)})QP_{(1)}\otimes P_{(2)}$. Using \Cref{con1} and \Cref{con2} we conclude that there exists $y$ such that   \[\Delta(Q)(Q\otimes\I) = Q\otimes y.\] Similarly using the second equality of \Cref{idenlocal} we get the existence of $\tilde y$ such that  \[(Q\otimes\I)\Delta(Q) = Q\otimes \tilde{y}.\] Noting that $(Q\otimes\I)\Delta(Q)(Q\otimes\I) = Q\otimes y = Q\otimes\tilde{y} $ we get $y=\tilde{y}$. 
 Using \Cref{remls} we see that $y$ is a two sided  group-like projection. 
\end{proof}
\subsection{Left coideal subalgebras with integrals and right group-like projections}\label{Sec_left_coideal}
\begin{definition}\label{def_coids}
Let $P\in\mathcal{A}$ be a right group-like projection. Consider 
\begin{align}
  \cN^l_P&=\{x\in \cA: (\I\otimes P)\Delta(x) = x\otimes P\} \label{eq:nlp}\\
  \cN^r_P&=\{x\in \cA: \Delta(x) (\I\otimes P) = x\otimes P\}\label{eq:nrp}\\
  \cN_P &= \{x\in\mathcal{A}:\Delta(x)(\I\otimes P) = (\I\otimes P)\Delta(x) = x\otimes P\} \label{eq:np}
\end{align}
Note that these are all left coideal subalgebras and $P\in\cN_P=\cN^l_P\cap \cN^r_P$. When we wish to refer to all three left  coideal subalgebras simultaneously we write $\cN_P^{-lr}$ (meaning the superscript can be blank, `$l$' or `$r$'); similarly, `$\cN_P^{lr}$' means `$\cN_P^l$ or $\cN^r_P$'.  

The smallest left coideal in $\cA$ containing $P$ will be denoted $\cV_P$ 
\begin{equation}
    \label{eq:n1}
    \cV_P = \{(\mu\otimes \id)(\Delta(P)):\mu\in\mathcal{A}^*\}.
  \end{equation}
\end{definition} 
With this description, $\cV_P$ does not immediately  seem to form an algebra;  this issue will be discussed later.

Our first observation concerning $\cN_P$ is as follows. 

\begin{proposition}\label{pr.min-cntr}
  Let $P\in \mathcal{A}$ be a right group-like projection and $\cN_P\subset \mathcal{A}$ the associated left coideal subalgebra, as in \Cref{eq:np}.

  Then, $P$ is a minimal central projection in $\cN_P$ and $\ker\left(\varepsilon|_{\cN_P}\right) = (1-P)\cN_P$. In particular if $\cN_{P_1} = \cN_{P_2}$ then $P_1 = P_2$.
\end{proposition}
\begin{proof}
  Applying $\varepsilon\otimes\id$ to the two equalities
  \begin{equation*}
    \Delta(x)(\I\otimes P) = x\otimes P = (\I\otimes P)\Delta(x),\ x\in \cN_P
  \end{equation*}
  shows that for $x\in \cN_P$ we have $xP = \varepsilon(x)P = Px$. In particular 
  \begin{equation*}
    xP=0\iff \varepsilon(x)=0 \iff Px=0. 
  \end{equation*}
  This implies both the centrality claim and the description of $\ker\left(\varepsilon|_{\cN_P}\right)$, while the minimality follows from this: indeed, we now know that $P\cN_P$ is one-dimensional. Assuming that $\tilde{P}$ is a group-like projection such that $\cN_P = \cN_{\tilde{P}}$ we get $P=\varepsilon(P)P = P\tilde{P} = \varepsilon(\tilde P)\tilde P = \tilde P$.
\end{proof}
Using \cite[Theorem 1]{skr-fin}  and \Cref{pr.min-cntr} we get  
\begin{corollary}
    The number of right group like projections in a finite dimensional semisimple or cosemisimple Hopf algebra is finite.
  
\end{corollary}

According to \Cref{pr.min-cntr} a right group-like projection $P$ has the following property:
\begin{equation}\label{eq:int1}
  xP = \varepsilon(x) P= Px,\ \forall x\in \cN_P. 
\end{equation}
This is the special case of the condition defining two sided integrals for left or right  coideals subalgebras as defined in \cite[Section 2]{Kopp}.

\begin{definition}\label{def.integral}
A {\it left integral} for a left coideal subalgebra $\cN\subset \mathcal{A}$ is an element  $\Lambda\in \cN$  satisfying $x\Lambda = \varepsilon(x)\Lambda$  
for all $x\in\cN$.
A {\it right integral} for a left coideal subalgebra $\cN\subset \mathcal{A}$ is an element  $\Lambda\in \cN$  satisfying $\Lambda x = \varepsilon(x)\Lambda$ for all $x\in\cN$.  
Finally, a   {\it two sided integral} is an element that is both a left and a right integral. 
We say that a left, right or two sided integral $\Lambda$ is counital if $\varepsilon(\Lambda) =1$.
\end{definition}
\begin{proposition}
  Let  $\cN\subset \cA$ be a left coideal subalgebra and $P$ be a group-like projection which is a counital two sided  integral for $\cN$. Then $\cN\subset \cN_P$. In other words $\cN_P$ is a maximal left coideal subalgebra for which $P$ is a counital two sided integral.

  Analogous statements hold for $\cN^l_P$ ($\cN^r_P$) and right (respectively left) integrals. 
\end{proposition}
\begin{proof}
  We only prove the two-sided statement; the rest is entirely parallel.

  It is enough to note that for all $x\in\cN$ we have $\Delta(x)\in\cA\otimes\cN$ and thus  
  $\Delta(x)(\I\otimes P) = x\otimes P =(\I\otimes P)\Delta(x)$. 
\end{proof}

The theory of left coideal subalgebras with integrals developed in \cite{Kopp} assumes the corresponding coideals are finite dimensional. Our first aim is to show that this condition holds automatically for $\cN_P$. Actually as was subsequently  proved in \cite{Kasp_coint} a left coideal subalgebra admitting a (generalized) non-zero integral is automatically finite dimensional. 
Let us first prove a technical result. 

\begin{lemma}\label{ideal_prop}Let $x,y\in\cA$.
The following conditions are equivalent: 
\begin{itemize}
    \item $\Delta(x)(\I\otimes y) = x\otimes y$;
    \item $(\I\otimes x)\Delta(y) = (S(x)\otimes\I)\Delta(y)$.
\end{itemize}

Similarly the following conditions are equivalent:
\begin{itemize}
    \item $(\I\otimes y)\Delta(x) = x\otimes y$;
    \item $\Delta(y)(\I\otimes x) = \Delta(y)(S^{-1}(x)\otimes\I)$.
\end{itemize}
\end{lemma}
\begin{proof}
Let us prove the equivalence of the first and the second bullet point of the lemma. 

If $x_{(1)}\otimes x_{(2)}y = x\otimes y$ then  $x_{(1)}\otimes x_{(2)}y_{(1)}\otimes x_{(3)}y_{(2)}  = x\otimes y_{(1)}\otimes y_{(2)}$ which implies that $y_{(1)}\otimes x y_{(2)} = S(x)y_{(1)}\otimes y_{(2)}$. 

Conversely, if $y_{(1)}\otimes x y_{(2)} = S(x)y_{(1)}\otimes y_{(2)}$ then $y_{(1)}\otimes x_{(1)} y_{(2)}\otimes x_{(2)} y_{(3)}  = S(x)y_{(1)}\otimes y_{(2)}\otimes y_{(3)}$ and then 
 $x_{(1)} y_{(2)}S^{-1}(y_{(1)})\otimes x_{(2)} y_{(3)}  = y_{(2)}S^{-1}(y_{(1)}) x\otimes y_{(3)}$ which shows that $\Delta(x)(\I\otimes y) = x\otimes y$.
\end{proof}

\begin{proposition}\label{cor1}Let $P$ be a right group-like projection and 
  $\cV_P\subseteq\cN^{-lp}$ be the left coideal defined in \Cref{def_coids}. Then $\cV_P$ is a
  \begin{itemize}
  \item right ideal in $\cN^l_P$; 
  \item left ideal in $\cN^r_P$; 
  \item bilateral ideal in $\cN_P$.   
  \end{itemize}
Moreover $\cV_P$ is idempotent as a (left, right or bilateral) ideal, i.e. $\cV_P^2=\cV_P$. 

In particular  $\cV_P = \cN^{-lr}_P$ if and only if $\I\in\cV_P$. In that case we also have
\begin{equation*}
  \cN^l_P = \cN_P = \cN^r_P. 
\end{equation*}
\end{proposition}
\begin{proof}
  That $\cV_P\subseteq \cN^l_P$ is a right ideal can be seen by applying $\mu\otimes\id$  to the fourth bullet point in \Cref{ideal_prop} with $y=P$, where   $\mu\in \cA^*$. The claim that $\cV_P\subseteq \cN_P^r$ is a left ideal can be seen similarly. $\cV_P$ is then both a left and a right ideal in $\cN_P=\cN_P^l\cap \cN_P^r$. 

  As for the claim that $\cV_P$ is idempotent, this follows from the fact that it contains $P$ hence so does  $\cV_P^2$, meaning that $\cV_P\subseteq \cV_P^2$. 

  Finally, if $\I\in\cV_P$ then $\cV_P$ is a (left, right, or bilateral) ideal containing the unit in $\cN_P^{-lr}$ and hence coincides with all of them.
\end{proof}

\begin{proposition}\label{pr.np-fd}
  Let $P\in \cA$ be a right group-like projection. Then, $\cV_P\subseteq \cN^{-lr}_P$ defined by \Cref{eq:n1} is faithful as a
  \begin{itemize}
  \item right ideal over $\cN^l_P$;    
  \item left ideal over $\cN^r_P$;
  \item left or right ideal over $\cN_P$.     
  \end{itemize}
\end{proposition}

Before moving on to the proof we record a consequence. 

\begin{corollary}\label{cor.fd}
  For any right group-like projection $P\in \cA$ the associated left coideal subalgebras $\cN^{-lr}_P$ are finite-dimensional.
\end{corollary}
\begin{proof}
  Indeed, by \Cref{pr.np-fd} they act faithfully on a finite-dimensional vector space $\cV_P$, and hence admit embeddings into the finite-dimensional algebra $\mathrm{End}(\cV_P)$.
\end{proof}

\begin{proof}[Proof of \Cref{pr.np-fd}]
To fix ideas and notation we will show that $\cV_P\subseteq \cN_P^r$ is faithful as a left ideal. 
  
According to \Cref{cor1} $\cV_P\subseteq \cN^r_P$ is an inclusion of left coideal subalgebras of $\cA$, and hence the multiplication map
  \begin{equation*}
    \cN^r_P\otimes \cV_P\to \cV_P 
  \end{equation*}
  is a morphism of left $\cA$-comodules. It induces a morphism
  \begin{equation*}
    \cN^r_P\to \cV_P\otimes \cV_P^*\cong \mathrm{End}(\cV_P)
  \end{equation*}
  in the same category of $\cA$-comodules, and hence its kernel must again be a left $\cA$-coideal. In other words, the (two-sided) ideal
  \begin{equation*}
    \cV_P^{\perp}:=\{x\in \cN^r_P\ |\ x\cV_P=0\}
  \end{equation*}
  is also a left $\cA$-coideal. The identity $x_{(1)}\varepsilon(x_{(2)})=x$ for $x\in \cV_P^{\perp}$ then shows that if $\cV_P^{\perp}$ is non-zero, then it contains elements $y$ such that $\varepsilon(y)\ne 0$.

  This leads to a contradiction: on the one hand we have $P\in \cV_P$ and hence $yP=0$, while on the other hand
  \begin{equation*}
    yP = \varepsilon(y)P\ne 0
  \end{equation*}
  by assumption. It follows that $\cV_P^{\perp}=\{0\}$, which is a rephrasing of the conclusion we were after.
\end{proof}
Knowing that $\cN_P$ is finite dimensional we can characterize the cases when   $\cV_P = \cN_P$. For this purpose  following property will be needed. 
\begin{definition}\label{def.frob}
A finite-dimensional algebra is {\it quasi-Frobenius} (or QF) if it is injective as a module over itself. 
\end{definition}

Using \cite[Lemma 2.1 and Theorem 2.2]{Kopp} we get
\begin{theorem}\label{qf_fun}
  Let $P\in\cA$ be a right group-like projection, $\cN_P$ the assigned left coideal subalgebra and $\cV_P\subset\cN_P$ the left coideal as defined in \Cref{def_coids}. Then $\cV_P = \cN_P$ if and only if $\cN_P$ is quasi-Frobenius.
\end{theorem}

We do not know an example where $\cV_P\ne \cN_P$, but note that the finite-dimensionality of $\cN_P$ together with the faithfulness result in \Cref{pr.np-fd} are not themselves sufficient to conclude this. 

\begin{example}\label{ex.ff}
  We will describe a finite-dimensional algebra $A$ with a proper idempotent ideal $I$ (i.e. $I^2=I$) that is both left and right faithful. 

  $A$ will be spanned by a projection $e$, its complementary projection $1-e$ and the basis $\{x,y,z\}$ for its Jacobson radical, subject to the following conditions:
  \begin{itemize}
  \item $x=ex(1-e)$;
  \item $y=(1-e)ye$;
  \item $yx=0$ and $xy=z$;
  \item $xz=zx=0=yz=zy$;    
  \end{itemize}
everything else follows (e.g. $x^2=0$ because $(1-e)e=0$). 
  
We set
\begin{equation*}
 I=\mathrm{span}\{e,x,y,z\} = Ae+eA = AeA, 
\end{equation*}
clearly an idempotent ideal. Checking faithfulness is a simple matter (left to the reader) of writing all five basis elements of $A$ as matrices for the multiplication action and observing that said matrices are linearly independent. In fact it suffices to check this for left multiplication, since the map interchanging $x$ and $y$ and fixing $e,z$ is an involutive anti-automorphism of $A$ fixing $I$.
\end{example}

\begin{definition}\label{def.wf}
  An algebra $R$ is {\it weakly finite} if $xy=1\Rightarrow yx=1$ for all $x,y\in M_n(R)$ and all positive integers $n$. 
\end{definition}

\cite[Theorem 6.1]{skr-proj}  shows that finite-dimensional left coideal subalgebras in weakly finite Hopf algebras are automatically Frobenius (and hence also QF). Therefor using \Cref{qf_fun} we get the next result.

\begin{corollary}\label{th.wf-hopf}
  Let $\cA$ be a weakly finite Hopf algebra (e.g. $\cA$ is left or right noetherian) and $P\in \cA$ a right group-like projection. Then $
    \cN_P = \cV_P$.
\end{corollary}
 
The results of \cite{Kopp} allows us also to deduce certain semisimplicity result  about $\cN_P$. Before formulating it let us prove the following proposition.

\begin{proposition}\label{coipr}
   Let $P\in\mathcal{A}$ be   right group-like  projection. Then the following are equivalent:
   \begin{itemize}
   \item $P$ is a two sided group-like projection;
   \item we have $S^2(\cN_P) = \cN_P$;
   \item we have $\Delta(P)\in S(\cN_P)\otimes \cN_P$.
   \end{itemize}
   Moreover, if $P$ is two-sided then $\cV_P = \cN^{-lr}_P$.  
\end{proposition}
\begin{proof}
Suppose that $P$ is a right group-like projection. Note that $S^2(P)$ is a right group-like projection and  $S^2(\cN_P) = \cN_{S^2(P)}$. In particular $S^2(\cN_P) = \cN_P$ if and only if $P = S^2(P)$ (c.f. \Cref{pr.min-cntr}) and this by \Cref{th.summary} holds if and only if $P$ is also a left group-like projection. 

Now, if $P$ is a two sided group-like projection then $S(P) = P$ and  it is easy to check that   $\Delta(P)\subset S(\cN_P)\otimes \cN_P$. Conversely if $P$ is a right group-like projection and $\Delta(P)\in S(\cN_P)\otimes \cN_P$  then $\Delta(P)(S(P)\otimes\I) = S(P)\otimes P = (S(P)\otimes\I)\Delta(P)$. Applying $\id\otimes\varepsilon$ we get $PS(P) = S(P)P = S(P)$ and we conclude that $P$ is a two sided group-like projection using  \Cref{th.summary}.

In order to prove the last claim of the proposition it is enough to show that $\I\in \cV_P$. Using Sweedler's notation $\Delta(P) = P_{(1)}\otimes P_{(2)}$ and we have $P_{(1)}\in S(\cV_P)$ and $P_{(2)}\in\cV_P$.  In particular $\I = S(P_{(1)})P_{(2)}\in\cV_P$ and we are done.
\end{proof}

Using \cite[Propositions 5.1 and 5.2]{Kopp} and \Cref{coipr} we get the next result.
\begin{theorem}\label{th.np-ss}
  If $P$ be a right group-like projection in $\mathcal{A}$, then $\cN_P$ is semisimple if and only if $\cN_P$ is a direct summand of $\mathcal{A}$ viewed as a $\cN_P$ two-sided module.

If $P$ is a two sided group-like projection in $\mathcal{A}$, then $\cN_P$ is semisimple.
\end{theorem}

\Cref{pr.np-fd} allows us to give the following partial characterization of coideal subalgebras of the form $\cN_P$.

\begin{proposition}\label{le.n=np}
  Suppose the coideal subalgebra $\cN\subset \mathcal{A}$ is semisimple. Then, $\cN=\cN_P$ for some right group-like projection $P$ and $\cN_P = \cV_P$. 
\end{proposition}
\begin{proof}
  The semisimplicity of $\cN$ means that $\ker(\varepsilon|_{\cN})$ is of the form $\cN(1-P)$ for a projection $P$ which is central in $\cN$, i.e. $\cN$ automatically has an integral $P$. We then have
  \begin{equation}
    \label{eq:n1nnp}
    \cV_P\subseteq \cN\subseteq \cN_P,
  \end{equation}
and it suffices to show that $\cV_P=\cN$ (for then $\I\in \cV_P$ and hence the inclusions in \Cref{eq:n1nnp} are all equalities). 

Finally, to prove $\cV_P=\cN$ recall that the former is a faithful ideal in the latter by \Cref{pr.np-fd}; since ideals in semisimple algebras are generated by central projections they cannot be (either left or right) faithful if proper. We thus have $\cV_P=\cN_P = \cN$, as claimed.   
\end{proof}

Finally, the results gathered thus far suffice to identify two-sided group-like projections as those giving rise to a certain class of coideal subalgebras. For locally compact quantum group the analog of the next result was obtained in \cite{Fall_Kasp}. It should also be compared with \cite[Proposition 5.2]{Kopp}. 

\begin{theorem}\label{th.2sd-s2}
  Let $\cA$ be an arbitrary Hopf algebra. The correspondence $P\mapsto \cN_P$ induces a bijection between the set of two-sided group-like projections and the set of semisimple coideal subalgebras of $\cA$ preserved by the squared antipode. 
\end{theorem}
\begin{proof}
  The map $P\mapsto \cN_P$ does indeed land in the set of semisimple coideal subalgebras fixed by $S^2$ by \Cref{th.np-ss}. It is one-to-one by \Cref{coipr} and onto by \Cref{le.n=np}.
\end{proof}

\subsection{Miscellaneous results}\label{subsec_misc}
In this subsection we shall prove some loosely related results concerning group-like projections, coideals, faithful flatness, Hopf subalgebras etc.   
 \begin{proposition}
  Let $P$ be a right group-like projection. Then $\cN_P$ forms a Hopf subalgebra if and only if $\Delta(P)\in \cN_P\otimes \cN_P$. In particular $P$ is two sided group-like projection in this case. 
\end{proposition}
\begin{proof}
 The leftward implication is obvious.  
  
  Suppose conversely that $\Delta(P)\in \cN_P\otimes\cN_P$. Then clearly \[\Delta(P)(P\otimes\I) = P\otimes P=(P\otimes\I)\Delta(P) \] i.e. $P$ is two sided.  Using \Cref{coipr} we get  $\cN_P = \{(\mu\otimes\id)(\Delta(P)):\mu\in\mathcal{A}^*\}$ and thus $S(\cN_P )= \{(\id\otimes\mu)(\Delta(P)):\mu\in\mathcal{A}^*\}\subset \cN_P$; similarly $S^{-1}(\cN_P) =   \{(\id\otimes\mu)(\Delta(P)):\mu\in\mathcal{A}^*\}\subset \cN_P$ and both containments  together yield  the equality $\cN_P = S(\cN_P)$. We conclude by noting that a left coideal subalgebra preserved by $S$ must be a Hopf algebra.
\end{proof}

\begin{proposition} \label{pr.p-cntr}
Let $P\in\cA$ be central group-like projection. Then $\cN_P$ is a normal left coideal subalgebra and $P$ is a two sided counital integral in $\cN_P$. Conversely, let $\cN$ be a normal left  coideal with a counital two sided  integral $P\in\cN$. Then $P$ is central in $\cA$ and $\cN = \cN_P$
\end{proposition}
\begin{proof}
  Suppose that $P$ is a central right group-like projection and $x\in\cN_P$. Then for every $a\in\mathcal{A}$ we have 
\begin{align*}\Delta(a_{(1)}xS(a_{(2)}))(\I\otimes P) &= a_{(1)}x_{(1)}S(a_{(4)})\otimes a_{(2)}x_{(2)}S(a_{(3)})P\\&= a_{(1)}x_{(1)}S(a_{(4)})\otimes a_{(2)}x_{(2)}PS(a_{(3)})\\&=a_{(1)}xS(a_{(4)})\otimes a_{(2)}PS(a_{(3)}) = a_{(1)}xS(a_{(2)})\otimes P\end{align*} and we see that $\cN_P$ is preserved by the adjoint action. Clearly $P$ is an integral for $\cN_P$. 

Conversely, suppose that $\cN$ is a normal coideal and $P\in\cN$ is an integral. Then for every $a\in\mathcal{A}$ we have \begin{equation}\label{aux1}a_{(1)}PS(a_{(2)})P = \varepsilon(a_{(1)}PS(a_{(2)}))P = \varepsilon(a)P\end{equation} and   similarly we get \begin{equation}\label{aux2}Pa_{(1)}PS(a_{(2)}) = \varepsilon(a)P.\end{equation} Viewing $P$ as a left multiplication operator on $\mathcal{A}$ we have
\begin{align*}\ww_{lr}(\I\otimes P)\ww_{lr}^{-1}(\I\otimes P) (a\otimes b) &= \ww_{lr} a_{(1)}\otimes PS(a_{2})Pb\\& = a_{(1)}\otimes a_{(2)}PS(a_{3})Pb\\& = a\otimes Pb\end{align*} 
where in the third equality we used \eqref{aux1}.
This shows that  $\ww_{lr}(\I\otimes P)\ww_{lr}^{-1}(\I\otimes P) = \I\otimes P$. Similarly we get $(\I\otimes P)\ww_{lr}(\I\otimes P)\ww_{lr}^{-1} = (\I\otimes P)$. Using  $(\I\otimes P)\ww_{lr}^{-1}(\I\otimes P) = \ww_{lr}^{-1}(\I\otimes P)$ we see that 
\[a_{1}\otimes PS(a_{(2)})Pb =a_{1}\otimes  S(a_{(2)})Pb  \] for every $a,b\in\mathcal{A}$. Applying $\varepsilon\otimes\id$ and putting $b=\I$  we get $PS(a)P = S(a)P$ for every $a\in\mathcal{A}$. Similar reasoning starting with  $(\I\otimes P)\ww_{lr}(\I\otimes P)\ww_{lr}^{-1} = (\I\otimes P)$ yields  $PaP = Pa$.  Thus $Pa = PaP = PS(S^{-1}(a))P = S(S^{-1}(a))P = aP$ for all $a\in\mathcal{A}$ and we see that  $P$ is central. Using \Cref{th.summary} we conclude that  $P$ is a two-sided group-like projection and using  \Cref{coipr} for the equality below we have  \[\cV_P\subseteq \cN\subseteq\cN_P=\cV_P\] and  $\cN = \cN_P$. 
\end{proof}

Central group-like projections for locally compact quantum groups were considered in \cite{KKS} in relations with open quantum subgroups. In our case we have:
\begin{corollary}
Let $P\in\mathcal{A}$ be a right group-like projection such that $\cN_P$ is normal. Then there is a Hopf algebra $\mathcal{B}$ and a Hopf surjection $\pi:\mathcal{A}\to\mathcal{B}$ such that $\cN_P = \{x\in\mathcal{A}:(\id\otimes\pi)\Delta(a) = a\otimes \I\}$.
\end{corollary}
\begin{proof}
Centrality of $P$ and  \Cref{th.summary} yields $S(P) = P$. We conclude by defining $\mathcal{B} = \mathcal{A}P$,  $\Delta_{\mathcal{B}}(b) = \Delta_{\mathcal{A}}(b) (P\otimes P)$, $S_\mathcal{B} = S_\mathcal{A}|_{\mathcal{B}}$ and $\pi(a) = aP$ for all $a\in\mathcal{A}$ and $b\in\mathcal{B}$.
\end{proof}

Let us recall that $\cN_P$ is a left coideal subalgebra containing a projection $P$ satisfying $Px = \varepsilon(x)P = xP$ for all $x\in\cN_P$. In the next proposition we get a   converse to this fact for two sided group-like projections, which is the Hopf algebraic analog of \cite[Theorem 4.1.]{Kasp_id_st}.
 \begin{proposition}\label{square}
Let $\cN\subset \mathcal{A}$ be a coideal subalgebra admitting a non-zero projection $Q\in\cN$  and a homomorphism $\mu:\cN\to\mathbbm{k}$ such that $Qx = xQ = \mu(x)Q$ for every $x\in\cN$.  If $S^2(Q) = Q$ then there exists a two sided group-like projection such that $\cN=\cN_P$ and we have $\cN = \{(\nu\otimes \id)(\Delta(Q)):\nu\in\mathcal{A}^*\}$. Moreover $Q$ is a left shift of $P$. 
\end{proposition}
\begin{proof} We denote $\pi = (\id\otimes\mu)\circ\Delta :\cN\to\mathcal{A}$. Note that  
 $\Delta(a)(\I\otimes Q) = \pi(a)\otimes Q$. It is easy to check that $\pi$ is an injective algebra homomorphism such that $\Delta\circ\pi = (\id\otimes\pi)\circ\Delta$. In particular $\pi(\cN)$ is a left coideal subalgebra and it contains a two sided group-like  projection $\pi(Q)$  (see \Cref{basic_lem}), thus $
 \cV_{\pi(Q)}\subset \pi(\cN)$. 
 Moreover for every $b\in\pi(\cN)$ there exists $c\in\mathcal{A}$ such that  $\Delta(b)(\I\otimes\pi(Q)) = c\otimes \pi(Q)$ and since $\varepsilon(\pi(Q))=\I$ we get $b=c$. This shows that $\pi(\cN)\subset \cN_{\pi(Q)}$. Since $\pi(Q)$ is two sided  we have $\cN_{\pi(Q)} =\cV_{\pi(Q)} = \pi(\cN)$ and thus $\cN = \{(\mu\otimes \id)(\Delta(Q)):\mu\in\mathcal{A}^*\}$. 
Using  $\Delta(Q)(\I\otimes Q) = \pi(Q)\otimes Q$ and  \Cref{lrshift} we can find 
a two sided group-like projection $P$ such that $\Delta(Q)(Q\otimes\I) = Q\otimes P$. In particular $P\in\cN$ and thus $\cN_P\subset \cN$. On the other hand since $Q\otimes Q = S^{-1}(P_{(1)})Q\otimes P_{(2)}$ we also get $Q\in\cN_P$ which implies that $\cN\subset \cN_P$ and we are done. 
\end{proof}

\begin{remark}\label{fflat1}
Using \cite{skr-proj} we conclude that if   $\cA$ is weakly finite Hopf algebra $\cA$ then $\cA$ is    $\cN_P$ -faithfully flat for every right group-like projection $P\in\cA$ and in particular  \begin{equation}\label{fflat}\cN_P = \{x\in\cA:(\id\otimes\pi)(\Delta(x)) = x\otimes\pi(\I)\}\end{equation} where $\pi:\cA\to\cC_{\cN_P}$ is the quotient map (see the first paragraph of \Cref{se.pre} for more explanations).  

 $\cN_P$-faithful  flatness is also guarantied if $\cN_P$  is semisimple, which is the case if e.g. $P$ is a two sided group-like projection (see \Cref{th.np-ss}).
But in general we do not know if $\cA$ is $\cN_P$ -faithfully flat,  and we do not know if  \Cref{fflat} holds.  

In order to relate in \Cref{fflat_prop} the right hand side of \Cref{fflat} to $\cN_P$  in possibly non $\cN_P$-faithfully flat case, we shall still denote  $\pi:\cA\to\cC_{\cN_P}$ and we denote  \[\cN_\pi = \{x\in\cA:(\id\otimes\pi)(\Delta(x)) = x\otimes\pi(\I)\}.\] Consider also the left $\cA$ module $\cA P$. Equipping  $\cA P$ with the coalgebra structure $xP\mapsto \Delta(xP)(P\otimes P)$ we see that $\cA P$ has a structure of a left module quotient coalgebra.

\end{remark}

\begin{proposition}\label{fflat_prop} Let $P\in\cA$ be a right group-like projection, $\cN_P$ the left coideal subalgebra assigned to $P$ and $\cC_{\cN_P}$, $\cA P$ the module quotient coalgebras described in \Cref{fflat1}. Let $\pi:\cA\to \cC_{\cN_P}$ be the canonical quotient maps. Then $\pi|_{\cA P}$ identifies $\cA P$ with $\cC_{\cN_P}$    and we have $\cN_{\pi} = \cN_P^r$. In particular if $\cN_P = \cV_P$ then $\cN_P = \cN_{\pi}$. 
\end{proposition}
\begin{proof}
Since $\I-P\in\cN_P^-$ we have $\pi(\I) =\pi(P)$. In particular, if $x\in\cA$ then $\pi(x) = x\pi(\I) = x\pi(P) = \pi(xP)$ and thus $\pi|_{\cA P}$ is onto. If in turn $\pi(xP) = 0$ then $0 = x\pi(P) = x\pi(\I) = \pi(x)$ and hence $x\in\cA\cN_P^{-}$. In particular $xP\in \cA\cN_P^{-}P= \{0\}$ and we see that $\pi|_{\cA P}$ is isomorphism of left $\cA$ modules. Finally, the equality  $\Delta(xP)(P\otimes P) = \Delta(x)(P\otimes P)$  shows that $\pi|_{\cA P}$ identifies $\cA P$ and $\cC_{\cN_P}$ as module quotient coalgebras. Under this identification the quotient map becomes $\pi(x) = xP$ which immediately implies that \[\cN_\pi = \{x\in\cA:\Delta(x)(\I\otimes P) = x\otimes P\}=  \cN_P^r.\] Using \Cref{cor1} we see  that $\cN_P = \cN_\pi$ if $\cN_P = \cV_P$. 
\end{proof}
\begin{remark}\label{weak_proj}
Let $P\in\cA$ be a projection in $\cA$ satisfying $\Delta(P)(P\otimes P) = P\otimes P$ and consider the left $\cA$ module $\cC_P=\cA P$. Let us define  $\Delta_{\cC_P}:\cC_P\to\cC_P\otimes\cC_P$ by $\Delta_{\cC_P}(x) = \Delta(x)(P\otimes P)$. It is easy to see that that $\Delta_{\cC_P}$  is coassociative. Moreover if $\varepsilon(P) = 1$ then $(\cC_P,\Delta_{\cC_P},\varepsilon|_{\cC_P})$ together with the map $\pi:\cA\to\cC$ given by $\pi(x) = xP$ is a left module quotient coalgebra. Clearly $\cN_\pi = \{x\in\cA:\Delta(x)(\I\otimes P) = x\otimes P\}$ thus  $P\in\cN_\pi$ if and only $\Delta(P)(\I\otimes P) = P\otimes P$. Moreover if  $P\in\cN_\pi$ then  $xP = \varepsilon(x)P$ for all $x\in\cN_\pi$.  

In what follows we shall show that $\cC_P$ can be identified with $\cC_{\cN_\pi}$ if $P\in\cA$ is a projection such that $\Delta(P)(\I\otimes P) = P\otimes P$. Clearly such a projection satisfies $\Delta(P)(P\otimes P) = P\otimes P$ and $\varepsilon(P) = 1$ and it makes sense to consider  a left module quotient coalgebra $\cC_P$ in the first place. In order to identify $\cC_P$ and $\cC_{\cN_\pi}$ we must show that  $\cA\cN_\pi^- = \cA(\I-P)$. The containment $\cA(\I-P)\subset \cA\cN_\pi^-$ is clear. For the converse containment we have $\cA\cN_\pi^- = \cA\cN_\pi^-(\I-P)\subset\cA(\I-P)$. 
\end{remark}
\begin{definition}\label{def_weak}
A projection  $P\in\cA$ satisfying and $\varepsilon(P) = \I$ is called
 \begin{itemize}
     \item  a right sided weak group-like projection if $\Delta(P)(P\otimes P) = P\otimes P$;
     \item a left sided weak group-like projection if $(P\otimes P)\Delta(P) = P\otimes P$;
     \item a weak group-like projection if $\Delta(P)(P\otimes P) =(P\otimes P) \Delta(P) = P\otimes P$. 
 \end{itemize} 
\end{definition}
Clearly a right (left) group-like projection is also a weak group-like projection. In \Cref{weak_sweedl} we shall see that the number of weak group-like projections in Sweedler's Hopf algebra is substantially larger the the number of right (left) group-like projections. 

Motivated by  \cite[Theorem 4.5]{KaspSol} we prove the next result.
\begin{proposition}
  Let $P\in\cA$ be a right sided weak group-like projection  $\Delta(P)(P\otimes P) = P\otimes P$ and $S(P) = P$. Then $P$ is a group-like projection.
\end{proposition}
\begin{proof}
 For the needs of the Sweedler's notation  we shall denote $P = \tilde{P}$ when necessary. Using  $P_{(1)}\tilde P\otimes P_{(2)}\tilde P  = P\otimes P$ we get  
 \[P_{(1)}\tilde P_{(1)}\otimes S(\tilde P_{(2)})S( P_{(2)})P_{(3)}\tilde P = P_{(1)}\otimes S(P_{(2)})P.\] Noting that the left hand side is equal $PP_{(1)}\otimes S(P_{(2)})P$ we conclude that (note that tildas are dropped)
 \[S(P_{(1)})S(P)\otimes S(P_{(2)})P = S(P_{(1)})\otimes S(P_{(2)})P \] Now using $S(P) = P$ we get   $\Delta^{\textrm{op}}(P)(\I\otimes P) = P\otimes P$.
 
 Since $\Delta^{\textrm{op}}(P)(P\otimes P) = P\otimes P$ we can apply the same reasoning to $\Delta^{\textrm{op}}$ to get $\Delta(P)(\I\otimes P) = P\otimes P$. Using \Cref{th.summary} we see that $P$ is a two sided group-like projection.
\end{proof}
\subsection{Semisimple Hopf algebras}\label{se.ss-coss}

Throughout the present section $\cA$ denotes a semisimple Hopf algebra over $\mathbbm{k}$.

\begin{lemma}\label{le.coid}
An arbitrary coideal subalgebra $\cN\subseteq \cA$ is of the form $\cN_P$ for some right group-like projection $P$. 
\end{lemma}
\begin{proof}
 Since $\mathcal{A}$ is semisimple, so is $\cN$ (e.g. \cite[Lemma 4.0.2]{bur-coid} or \cite[Theorem 5.2]{skr-proj}). Now simply apply \Cref{le.n=np}. 
\end{proof}

All in all, the conclusion is that right group-like projections are an alternate manifestation of coideal subalgebras, at least in semisimple Hopf algebras, as the following one-sided analogue of \Cref{th.2sd-s2} confirms.

\begin{theorem}\label{th.proj-coid}
  Let $\cA$ be a semisimple Hopf algebra. The maps
  \begin{equation*}
   \cN\mapsto P_\cN:=\text{the central projection satisfying }\ker\left(\varepsilon|_{\cN}\right) = (\I-P)\cN 
  \end{equation*}
and $P\mapsto \cN_P$ are mutually inverse bijections between the set of coideal subalgebras of $\cA$ and its set of right group-like projections. 
\end{theorem}
\begin{proof}
  \Cref{le.coid} verifies that
  \begin{equation*}
    \cN\mapsto P_{\cN}\mapsto \cN_{P_{\cN}}
  \end{equation*}
is the identity. On the other hand, consider a right group-like projection $P$. Its associated coideal subalgebra $\cN=\cN_P$ has an attached projection $Q=P_{\cN_P}$. \Cref{le.coid} once more shows that $\cN_Q={\cN}=\cN_P$, meaning that $Q=P$ by \Cref{coipr}.   
\end{proof}

\begin{remark}
  Using \cite[Theorem 3]{Radford} and \Cref{th.summary} we see that   right group-like projections are two sided group-like projections if $\cA$ in a simple Hopf algebra over a field of characteristic $p>\dim(H)^2$.

  In fact, more generally, if the Hopf algebra is simultaneously semisimple and cosemisimple then its antipode is an involution \cite{eg}; it then again follows from \Cref{th.summary} that right group-like projections are two-sided. 
\end{remark}

\begin{remark}\label{rem_bijc} Let $\cA$ be a finite dimensional Hopf algebra and $\cN$ a left coideal subalgebra. Let $\pi:\cA\to\cC_\cN$ be the assigned left module  quotient  coalgebra as described in \Cref{fflat1}. Then   $\pi^*(\cC^*)\subset\cA^*$ forms a left coideal subalgebra. 
Using \cite[Corollary 6.5]{skr-proj} we see that  $\cN\mapsto\pi^*(\cC_{\cN}^*)$  establishes a bijective correspondence between left coideal subalgebras of $\cA$ and $\cA^*$. Let us note that $\pi^*(\cC_{\cN}^*)$ is of the form $\cN_{\tilde{P}}$ for some right  group like projection $\tilde{P}\in\cA^*$ if and only if $\pi^*(\cC_{\cN}^*)$ admits an a two sided counital integral in the sense of \Cref{def.integral}. Indeed, in this case \[\cV_{\tilde{P}}\subseteq \pi^*(\cC_{\cN}^*)\subseteq \cN_{\tilde{P}}=\cV_{\tilde{P}}\] where the in last equality we use \Cref{th.wf-hopf}. Clearly $\pi^*(\cC_{\cN}^*)$ admits an integral if and only if there exists a functional $\psi\in\cC_{\cN}^*$ such that $(\psi\otimes\id)(\Delta_{\cC_{\cN}}(x)) = (\id\otimes\psi)(\Delta_{\cC_{\cN}}(x)) = \psi(x)\pi(\I)$ for all $x\in\cC_{\cN}$. 

Suppose that  $\cA$ is a semisimple and cosemisimple Hopf algebra. Then $\cA^*$ is also semisimple and cosemisimple.  Since in this case the antipode is involutive, then using \Cref{th.proj-coid} and the above considerations  we get a bijective correspondence between   group-like projections in $\cA$ and $\cA^*$. 
\end{remark}


\section{Taft algebras}\label{se.tft}

The purpose of the present section is to classify the left coideal subalgebras of a certain family of Hopf algebras, identifying among them those that correspond to right group-like projections via the construction $P\mapsto \cN_P$ from \Cref{eq:np}. 

Following \cite{tft}, we have

\begin{definition}\label{def.tft}
  Let $n$ be a positive integer and $\omega\in \kk$ a primitive $n^{th}$ root of unity. The {\it Taft Hopf algebra} $H_{n^2}$ is the algebra $\kk[x]/(x^n)\rtimes \kk[G]$, where $G=\mathbb{Z}/n\mathbb{Z}$ is generated by $g$ and its action on the truncated polynomial ring $\kk[x]/(x^n)$ is by
  \begin{equation*}
    gxg^{-1} = \omega x. 
  \end{equation*}
  The coalgebra structure is given by
  \begin{equation*}
    \Delta(g) = g\otimes g,\quad \Delta(x)=x\otimes 1 + g\otimes x. 
  \end{equation*}
\end{definition}

These Hopf algebras appear frequently in the literature, being $\omega$-deformed analogues of the enveloping algebra of the Borel subalgebra $\mathfrak{b}\subset\mathfrak{sl}_2$. See e.g. \cite{majid,pw-var-mod,pv-ideals}. They were originally introduced in \cite{tft} as part of a larger family meant to show that finite-dimensional Hopf algebras can have antipodes of arbitrary even order (for $H_{n^2}$ the antipode order is $2n$). 

The notation $H_{n^2}$ is meant to remind the reader that $H$ so defined is $n^2$-dimensional, for instance with a basis consisting of
\begin{equation*}
  x^ig^j,\ 0\le i,j\le n-1.
\end{equation*}
Throughout this section we focus on the case when $n$ is prime, this assumption being henceforth implicitly in place unless specified otherwise. Moreover, since we are assuming that $\kk$ contains primitive $n^{th}$ roots of unity, the characteristic of $\kk$ is coprime to $n$ and hence $\kk G$ is semisimple.

Let $\cN\subset H=H_{n_2}$ be a proper, non-trivial left coideal subalgebra. $\cN$ further maps through
\begin{equation}\label{eq:nhg}
  \cN\to H\to \kk G
\end{equation}
to a coideal subalgebra of $\kk G$. The latter are simply the subalgebras of the form $\kk G'$ for subgroups $G'\le G$, and hence are automatically semisimple. On the other hand, the kernel of \Cref{eq:nhg} consists of nilpotent elements. In conclusion, the kernel of \Cref{eq:nhg} is precisely the Jacobson radical $J(\cN)$ while its image is the semisimple quotient $\cN_{ss}$ of $\cN$. 

We have two possibilities: 
\begin{itemize}
\item $\cN$ is semisimple, meaning that \Cref{eq:nhg} is one-to-one;
\item it is not, i.e. $J(N)$ is non-zero. 
\end{itemize}
$H$ is naturally $\bZ$-graded by
\begin{equation*}
  \deg g=0,\ \deg x=1. 
\end{equation*}
We denote the degree-$d$ component of $H$ by $H_d$. Equipped with the grading just defined $H$ is a graded Hopf algebra in the sense of \cite[Definition 10.5.11]{mont}: it is a graded algebra as well as a graded coalgebra, i.e.
\begin{equation*}
  \Delta(H_d)\subset \bigoplus_i H_i\otimes H_{d-i}. 
\end{equation*}
We will often refer to the degree of a possibly-non-homogeneous element $y$ of $H$, meaning the largest degree $i$ of an element $x^i g^j$ appearing in an expansion of $y$ in the basis $\{x^ig^j\}$. We write $\deg(y)$ for the degree in this sense, as well as $\deg'(y)$ for the {\it smallest} degree of a summand of $y$. The {\it free term} of $y$ is the sum of all of its degree-zero summands; equivalently, it is the image of $y$ through the surjection $H\to \kk G$. It is non-zero precisely when $\deg'(y)=0$. 

We now start to analyze the structure of the coideal subalgebras of $H$.

\begin{lemma}\label{le.lrg-deg}
  Let $n\ge 2$ be an arbitrary positive integer and $\cN\subset H=H_{n^2}$ a non-trivial proper left coideal subalgebra.
  
  If $\cN$ contains a non-zero element $y$ with no free term then it contains, for each $1\le k\le n-1$, an element of the form $x^kg^i$ for some $0\le i\le n-1$ (possibly dependent on $k$). 
\end{lemma}
\begin{proof}
  Let
  \begin{equation*}
    y=y_{d'}+\cdots +y_d
  \end{equation*}
  where $y_i\in H_i$ and 
  \begin{equation*}
    d=\deg(y),\ 0<d'=\deg'(y). 
  \end{equation*}
  Because $\cN$ is a left coideal subalgebra the elements
  \begin{equation*}
    \varphi(y_{(1)})y_{(2)},\ \varphi\in H^*
  \end{equation*}
  are all contained in $\cN$. Choosing $\varphi$ in the summand $H_1^*$ of $H^*$ will produce an element $z\in \cN$ with $\deg'(z)=\deg'(y)-1=d'-1$. Applying this repeatedly if necessary we obtain elements $z\in \cN$ with $\deg'(z)=1$. We thus discard the symbol `$z$' and assume $d'=1$ throughout the rest of the proof.

  Writing $y_1$ as $\sum a_j xg^j$ for $0\le j\le n-1$ we have
  \begin{equation*}
    y_{1(1)}\otimes y_{1(2)} = \sum a_j \left(xg^j\otimes g^j+g^{j+1}\otimes xg^j\right)
  \end{equation*}
  Applying to this element a functional of the form $\varphi\otimes \id$ for $\varphi\in H^*_0$ annihilating all $g^\ell$ save for a single $g^{j+1}$ with $a_j\ne 0$ we obtain an element $z\in \cN$ with $\deg'(z)=1$ and $z_1$ being a single element $xg^j$ of the standard basis of $H$.

  We now have
  \begin{equation*}
    x^{n-1}g^{j(n-1)} = z_1^{n-1} = z^{n-1}\in \cN,
  \end{equation*}
  i.e. $\cN$ contains monomials of top degree $n-1$. We re-purpose `$j$' and simply write $x^{n-1}g^j\in \cN$. The $H_{n-2}\otimes H_1$-component of $\Delta(x^{n-1}g^j)\in H\otimes \cN$ is
  \begin{equation*}
    (1+\cdots+\omega^{n-2})x^{n-2}g^{j+1}\otimes xg^j = -\omega^{n-1} x^{n-2}g^{j+1}\otimes xg^j\ne 0
  \end{equation*}
so the right hand tensorand $xg^j$ belongs to $\cN$. Finally, the powers of this element satisfy the desired conclusion. 
\end{proof}

This will allow us to better understand the coideal subalgebras falling in the second of the two qualitative classes noted in the discussion preceding \Cref{le.lrg-deg}.

\begin{proposition}\label{pr.nss}
  The non-semisimple coideal subalgebras $\cN\subset H$ are $\cN_{d,x}$ for various divisors $d|n$, defined by
  \begin{equation}\label{eq:nd}
    \cN_{d,x}:=\text{algebra generated by }g^d\text{ and }x.
  \end{equation}
\end{proposition}
\begin{proof}
It is immediate that $\cN_{d,x}$ is indeed a coideal subalgebra (it is the algebra generated by the coideals $\mathrm{span}\{1,x\}$ and $\kk g^d$). We now have to show that an arbitrary (non-semisimple) $\cN$ is of the form $\cN_{d,x}$ for some $d|n$. 

First, we identify $d$: the intersection $\cN\cap \kk G$ is of the form $\kk G'$ for some subgroup $G'\le G$, and we set $d=|G/G'|$. Now let $y\in \cN$ be an element of degree $\ell>0$ whose maximal-degree component $y_\ell$ contains a scalar multiple of some $x^{\ell} g^j$. 

Choose some $\varphi\in H^*$ that annihilates all $x^i g^j$ except for $i=\ell$. Applying $\varphi\otimes\id$ to $\Delta(y)$ produces a non-zero scalar multiple of $g^j$, which must belong to $\cN$ because the latter is a left coideal. Our definition of $d$ entails $d|j$. In conclusion, for all $y\in \cN$ the top-degree component $y_{\ell}$ belongs to $\cN_{d,x}$; by induction on degree we obtain $\cN\subseteq \cN_{d,x}$.

Conversely, $\cN$ contains $G'=\{g^{kd}\ |\ k\in \bZ\}$ as well as $xg^j$ for some $j$ by \Cref{le.lrg-deg}. We already know from the preceding discussion that $d|j$, so some element $(xg^j)(g^d)^k\in \cN$ equals $x$. This proves the opposite inclusion $\cN\supseteq \cN_{d,x}$ and finishes the proof.
\end{proof}

So far, by way of classification we have 

\begin{proposition}\label{th.tft-clsf}
  Let $n\ge 2$ be a positive integer. The left coideal subalgebras of $H_{n^2}$ fall into two classes:
  \begin{itemize}
  \item semisimple ones isomorphic to $\kk G'$ through \Cref{eq:nhg} for subgroups $G'\le G$, in bijection with right group-like projections; 
  \item non-semisimple, equal to one of the coideal subalgebras $\cN_{d,x}$ in \Cref{eq:nd} for divisors $d|n$.    
  \end{itemize}  
\end{proposition}
\begin{proof}
  This is a consequence of \Cref{pr.nss} together with the fact that a semisimple coideal subalgebra is of the form $\cN_P$ for a unique right group-like projection $P$ by \Cref{le.n=np}.
\end{proof}
\begin{remark}
Note that the non-semisimple coideals do not correspond to group-like projections. Indeed, suppose that $P\in H_{n^2}$ is a group-like projection which is an integral  for $\cN_{d,x}$. Then $Px^{n-1}=0$ and hence $P $ is of the form $yx$ for some $y\in H_{n^2}$. But this implies $\varepsilon(P) = 0$ - contradiction. 
\end{remark}
It remains to give a more concrete description of the semisimple coideal subalgebras, forming (as we will see) the bulk of the classification in \Cref{th.tft-clsf}.

\begin{proposition}\label{th.ss-clsf}
  Let $n\ge 2$ be an arbitrary positive integer.  The semisimple left coideal subalgebras of $H=H_{n^2}$ are of one of two types:
  \begin{itemize}
  \item $\kk G'\subset \kk G\subset H$ for some subgroup $G'\subset G$;
  \item $\cN_{\beta}:=\text{algebra generated by }g+\beta x$ for some $\beta\in \kk^{\times}$.    
  \end{itemize}
\end{proposition}
\begin{proof}
  Clearly, all subgroups $G'$ give rise to a Hopf subalgebra $\kk G'$ of $H$ as in the first bullet point above. We thus focus on showing that all {\it other} semisimple coideal subalgebras are of the form $\cN_{\beta}$ (noting also that all $\cN_{\beta}$ are indeed coideal subalgebras, as is easily seen).

  Suppose $\cN$ is a coideal subalgebra as in the statement. The composition \Cref{eq:nhg} identifies $\cN$ with some coideal subalgebra of $\kk G$, say $\kk G'$ for a subgroup $G'\le G$ of order $d|n$.

  {\bf Step 1: $G'=G$.}

  For each $s\in G'$ the subalgebra $\cN$ contains a unique element $s+\psi(s)$ with free term $s$ (i.e. such that $\psi(s)$ has no free term). Our assumption that $\cN$ is not $\kk G'$ means that $\psi:G'\to J(H)$ is not identically zero, and the conclusion we want to reach is that in this case $G'=G$ and $\psi(g)=\beta x$ for some scalar $\beta\ne 0$. 

  Let $s=g^{kd}\in G'$ be an element with $y:=\psi(s)\ne 0$. There are two possibilities to consider:
  
  {\bf (a): $\deg'(y)>1$.} Then, applying $(\varphi\otimes \id)\Delta$ to $y$ for some element $\varphi\in H^*_1$ will produce a non-zero element in $\cN\cap J(H)$, which we are assuming cannot exist. This leaves

  {\bf (b): $\deg'(y)=1$.} Suppose $y_1$ contains a non-zero scalar multiple of $xg^j$ for $j\ne kd-1$. We then have
  \begin{equation*}
    \Delta(y_1) \sim xg^j\otimes g^j+ g^{j+1}\otimes xg^j+\cdots,\ g^{j+1}\ne g^{kd}
  \end{equation*}
  where $\sim$ means `scalar multiple of' and $\cdots$ signifies the terms resulting from applying $\Delta$ to other monomials $\ne xg^j$ appearing in $y_1$.
  
  An application of $\varphi\otimes\id$ to $\Delta(g^{kd}+y)$ for $\varphi\in H_0^*$ dual to $g^{j+1}$ yields a non-zero element $z\in \cN$ with no free term (and $z_1\in \kk xg^j$). This contradicts the semisimplicity of $\cN$, meaning that $y_1\in \kk xg^{kd-1}$.
  
  Now, however, we have
  \begin{equation*}
    \Delta(y_1)\sim xg^{kd-1}\otimes g^{kd-1}+ g^{kd}\otimes xg^{kd-1},
  \end{equation*}
and setting $\varphi$ dual to $xg^{kd-1}$, $(\varphi\otimes\id)\Delta(y)$ will be a non-zero element of $\cN$ whose free term is (in the span of) $g^{kd-1}$. Since free terms of elements of $\cN$ are in $G'=\langle g^d\rangle$, we have $d|kd-1$ and hence $d=1$. In conclusion we have $G'=G$, as announced earlier. 

{\bf Step 2: finishing the proof.} Now denote $\psi(g)=g+y$. Since we are assuming $\cN\ne \kk G$ we have $\deg(y)\ge 1$. 

  Let $\varphi\in H^*$ be the dual element to $g$. The free term of
  \begin{equation*}
    (\varphi\otimes \id)\circ \Delta (g+y) 
  \end{equation*}
  is $g$, and since this element belongs to $\cN$ it must again be equal to $g+y$. It follows that the degree-$d$-component $y_d$ has the property that the $H_0\otimes H_d$-component of $\Delta(y)$ is in $\kk g\otimes H_d$, and hence $y_d\in \kk x^d g^{1-d}$. In other words, we have
  \begin{equation*}
    y=\psi(g)=\sum_{d>0} \alpha_d x^d g^{1-d}. 
  \end{equation*}
  For each term $y_d= \alpha_d x^d g^{1-d}$ the left hand tensorand in the $H_1\otimes H_{d-1}$-component of $\Delta(y_d)$ belongs to the span of $x$, and hence applying $\varphi\otimes \id$ to $\Delta(y)$ with $\varphi\in H_1^*$ dual to $x$ produces a non-zero element of degree $\deg(y)-1$ whose free term is $1$ if $\alpha_1\ne 0$ and $0$ otherwise.

  Since $\cN$ has no elements with no free terms the second possibility is ruled out. On the other hand, its only element with free term $1$ is the unit, so $\deg(y)=1$. This concludes the proof that $y=\beta x$ for some $\beta\in \kk^{\times}$.
\end{proof}

\begin{corollary}\label{cor.aut}
  For $n\ge 2$ the orbits of the coideal subalgebras of $H=H_{n^2}$ under the automorphism group of $H$ are as follows:
  \begin{itemize}
  \item a fixed point $\kk\langle g^d\rangle$ for each divisor $d|n$;
  \item a fixed point $\cN_{d,x}$ for each divisor $d|n$;
  \item the orbit consisting of all $\cN_{\beta}$ for $\beta\in \kk^{\times}$.    
  \end{itemize}
\end{corollary}
\begin{proof}
  This follows from the classification in \Cref{th.tft-clsf}, the description provided by \Cref{th.ss-clsf}, and the fact that $x\mapsto \beta x$ extends to an automorphism of $H$ for every $\beta\in \kk^{\times}$.
\end{proof}

The following result is the analogue for right group-like projections. For $\beta\in \kk$ we denote by $P_{\beta}$ the idempotent
\begin{equation}\label{pbeta}
  \frac 1n \sum_{i=0}^{n-1} (g+\beta x)^i. 
\end{equation}

\begin{corollary}\label{cor.pr-orbit}
  Let $n\ge 2$. The orbits of right group-like projections under the automorphism group of $H_{n^2}$ are the singletons
  \begin{equation*}
    \left\{\frac 1d \sum_{i=0}^{\frac nd-1}g^{di}\right\}\text{ for }\; d|n
  \end{equation*}
and $\{P_{\beta}\ |\ \beta\in \kk^{\times}\}$.  
\end{corollary}
Note that the right group-like projection $P_\beta$ is not preserved by the antipode if $\beta\in\kk^\times$. This answers  the question of M.~Landstad and A.~Van Daele formulated in \cite{LV} where they in  particular have shown that it cannot happen in algebraic quantum groups, see  \cite[Proposition 1.6]{LV}.

As a separate matter, {\it every} automorphism of $H_{n^2}$ is of the form $x\mapsto \beta x$:

\begin{proposition}\label{pr.autos}
  Let $n\ge 2$ be an arbitrary positive integer. The map $\beta\mapsto \varphi_{\beta}$ where
  \begin{equation*}
    \varphi_{\beta}:x\mapsto \beta x,\ g\mapsto g
  \end{equation*}
  is an isomorphism between $\kk^{\times}$ and the group of automorphisms of $H_{n^2}$.
\end{proposition}
\begin{proof}
  We have already observed that all $\varphi_{\beta}$ are indeed automorphisms; it remains to show that there are no others. To that end, let $\varphi$ be an automorphism of $H=H_{n^2}$. Since $H$ is generated by $g$ and $x$, it will be enough to describe the images of these two elements. 

Since $G$ is the group of group-like elements, $\varphi$ restricts as an automorphism of $G$. The component
  \begin{equation*}
    H_{(1)}:=\{x\in H\ |\ \Delta(x)\in H\otimes \kk G+\kk G\otimes H\}
  \end{equation*}
  of the coradical filtration \cite[$\S$5.2]{mont} of $H$ is
  \begin{equation*}
    H_0\oplus H_1 = \kk G\oplus\mathrm{span}\{xg^i\ |\ 0\le i\le n-1\}. 
  \end{equation*}
  For $s,t\in G$ denote
  \begin{equation*}
    P_{s,t}:=\{x\in H\ |\ \Delta(x) = x\otimes s+t\otimes x\},
  \end{equation*}
  i.e. the space of {\it $(s,t)$-primitive elements} \cite[$\S$5.4]{mont}. According to \cite[Theorem 5.4.1]{mont} $H_{(1)}$ is the direct sum of $\kk G$ and all $P'_{s,t}$ where the latter are arbitrary spaces satisfying
  \begin{equation*}
    P_{s,t} = \kk (s-t)\oplus P'_{s,t}. 
  \end{equation*}
  We have $xg^i\in P_{g^i,g^{i+1}}$, so $P_{1,g}=\mathrm{span}\{1-g,x\}$ and $P_{1,s}$ is one-dimensional for all $g\ne s\in G$.

  Since $\varphi$ permutes the non-trivial elements $s\in G$ it also permutes the spaces $P_{1,s}$, and hence it must fix the only element (namely $g$) whose associated primitive space $P_{1,g}$ is two-dimensional. In conclusion, we have $\varphi(g)=g$.

  As for $\varphi(x)$, it must be both an element of $P_{1,g}=\mathrm{span}\{1-g,x\}$ and an $\omega$-eigenvector for conjugation by $g$. The intersection of these two spaces is precisely $\kk x$, meaning that $\varphi(x)=\beta x$ for some $\beta\in \kk^{\times}$.
\end{proof}
\begin{remark}Let us consider the Taft Hopf algebra over the field of complex numbers. Given $z\in\CC$, $|z|=1$ there exists a uniquely determined $*$-structure on $H_{n^2}$ such that $g^* = g$ and $x^* = zx$.

Reasoning as in \Cref{pr.autos}  we can show that there are no other $*$-structures on $H_{n^2}$.  

Let us equip $H_{n^2}$ with the  $*$-structure   corresponding to $z=1$. Then for $\beta\in\RR$ $P_\beta\in H_{n^2}$ defined  \Cref{pbeta}  is  self adjoint right  group-like projection which is  not  left; this observation  strengthens the answer to the Landstad - Van Daele question. 
\end{remark}
\begin{remark}\label{weak_sweedl}
Let us consider the Sweedler's Hopf algebra, $H = H_{4}$. It has a basis $\{\I, g, x, y\}$ where $y = gx$, $g^2 = \I$, $gxg = -x$. Due to \Cref{cor.pr-orbit} the right group-like projections in $H$ are  of the form  $\I$, $\frac{1}{2}(\I+g)+\beta x$ where $\beta\in\kk$. By direct computations it  can be checked that weak group-like projections in $H_4$ (see \Cref{def_weak})  are of the form  $\I$ and  $  \frac{1}{2}(\I+g)+\beta x + \delta y$ for arbitrary $\beta,\delta\in\kk$. In particular the space of classes of weak group-like projections is infinite and can be identified with $\{\I\}\cup\{\frac{1}{2}(\I+g)\}\cup \CC \mathbb{P}^1$ where $\CC \mathbb{P}^1$ is $1$-dimensional complex projective space. 
\end{remark}

\section{Open problems}
We formulate  a list of questions  motivated by the results of the present paper.
\begin{enumerate}
\item Due to the result of \cite{LV} there does not exist a right group-like projection in an algebraic quantum group which is not two-sided. Can we prove this result for semisimple or cosemisimple Hopf algebras?  Let us note that for cosemisimple Hopf algebras it was recently proved under the assumption that $\cV_P = \cN_P$, see \cite[Remark 4.8]{Kasp_coint}.
\item Is there a finite dimensional Hopf algebra which has infinitely many classes of left coideal subalgebras?
\item Is there a finite dimensional Hopf algebra which has infinitely many two sided group-like projections?
    \item Is there a left coideal subalgebra $\cN$ with a two sided counital  integral $P$ such that $\cN$ is not quasi-Frobenius?
    \item Is there a right group-like projection $P$ such that $\cN_P$ is not semisimple? Is there a right group-like projection in $\cA$ such that $\cA$ is not faithfully flat over $\cN_P$?
\end{enumerate}


\bigskip

\subsection*{Acknowledgments}The authors are grateful to the anonymous referee for paying our attention to the paper of M. Koppinen, \cite{Kopp}. 
A.C. was partially supported by NSF grant DMS-1801011. PK was partially supported by the NCN (National Center of Science) grant
 2015/17/B/ST1/00085.


\bibliographystyle{plain}

\end{document}